\documentclass[11pt,reqno]{amsart}
\usepackage{amsmath, amsthm, amsfonts}

\usepackage[top=29mm,bottom=29mm,left=30mm,right=30mm]{geometry}

\usepackage{color}

\makeatletter
\def\section{\@startsection{section}{1}%
\z@{1\linespacing\@plus\linespacing}{1\linespacing}%
{\bf\centering}}
\def\subsection{\@startsection{subsection}{0}%
\z@{\linespacing\@plus\linespacing}{\linespacing}%
{\bf}}
\def\subsubsection{\@startsection{subsubsection}{0}%
\z@{\linespacing\@plus\linespacing}{\linespacing}%
{\bf}}
\makeatother

\makeatletter
\@addtoreset{equation}{section}
\makeatother

\newtheorem{theorem}{Theorem}[section]

\newtheorem{lemma}[theorem]{Lemma}

\newtheorem{remark}[theorem]{Remark}
\theoremstyle{definition}

\def\RR{\mathbb{R}}\def\R{\mathbb{R}}

\def\PP{\mathbb{P}}
\def\EE{\mathbb{E}}

\def\al{\alpha}
\def\be{\beta}
\def\ga{\gamma}
\def\de{\delta}

\def\ep{\varepsilon}

\def\la{\lambda}
\def\sk{\smallskip}
\def\mk{\medskip}

\def\ov{\overline}

\def\rd{\mathrm{d}}  

\newcommand{\bra}[1]{\left\lbrace#1\right\rbrace}

\DeclareMathOperator\dimh{dim_H} 
\DeclareMathOperator\dimp{dim_P} 

\begin{document}
\title
{Uniform Hausdorff dimension result for the inverse images of stable L\'evy processes}
\author{Renming Song\ \and Yimin Xiao \and Xiaochuan Yang}

\address{Department of Mathematics, University of Illinois, Urbana, IL 61801, USA}
\email{rsong@illinois.edu}

\address{Dept. Statistics \& Probability, Michigan State University, 48824 East Lansing, MI, USA}
\email{xiaoyimi@stt.msu.edu}
\email{yangxi43@stt.msu.edu, xiaochuan.j.yang@gmail.com}

\thanks{\emph{Key-words}: Stable processes, Inverse image, Hausdorff dimension
 \\ \medskip
\noindent
2010 {\it MS Classification}: Primary 60J75, 60G52, 60G17, 28A80  \\
\noindent
Research of Renming Song was supported in part by the Simons Foundation (\# 429343, Renming Song).
Research of Yimin Xiao was supported in part by the NSF grant DMS-1607089.
}

\begin{abstract}
We establish a uniform Hausdorff dimension result for the inverse image sets of real-valued strictly
$\alpha$-stable L\'evy processes with $1< \al\le 2$.
This extends  a theorem of Kaufman \cite{Kaufman85}
for Brownian motion. Our method is different from that of \cite{Kaufman85}
and depends on covering principles for Markov processes.

\end{abstract}

\maketitle

\baselineskip 0.5 cm

\bigskip \medskip

\section{Introduction}

Let $X=\{X(t), t\ge 0, \PP^x\}$ be a real-valued
strictly $\al$-stable L\'evy process with $\al \in (0, 2]$. Its characteristic 
exponent is given by, for $\xi\in\RR$,
\begin{equation*}
-\log\EE^0[e^{i\xi X(1)}] = \left\{\begin{array}{ll} \sigma^\al |\xi|^\al 
\Big(1- i\be\tan \frac {\pi \al} 2\mbox{sgn} \xi \Big), \ &\hbox{ if } \alpha \ne 1;\\
\sigma|\xi|, &\hbox{ if } \alpha = 1
\end{array}
\right.
\end{equation*}
with some constants $\sigma>0$ and  $\be\in[-1,1]$ which are respectively the 
scale parameter and the skewness parameter. Throughout $\log=\log_e$ denotes 
the natural logarithm. Notice that, in the case of $\alpha = 1$, $X$ is a 
symmetric Cauchy process. When $\alpha = 2$, $X$ is a (scaled) Brownian motion. 
For $0 < \alpha < 2$, $X$ shares the properties of self-similarity, independence  
and stationarity of increments, with Brownian motion, but it has heavy-tailed 
distributions and its sample functions are discontinuous. As such, stable L\'evy 
processes form an important class of Markov processes. Many authors have studied 
the asymptotic and sample path properties of L\'evy processes. We refer to the 
monographs \cite{Bertoin96book} and \cite{Sato13} for systematic accounts on 
L\'evy processes, and to \cite{Taylor86,Xiao04survey} for information on
their fractal properties.

\sk
This note is concerned with a uniform Hausdorff dimension result, Theorem \ref{theo},
for the inverse images of real-valued strictly $\al$-stable L\'evy processes and is 
motivated by the following results of Hawkes \cite{Hawkes71b} and Kaufman \cite{Kaufman85}.

Hawkes \cite{Hawkes71b} considered the Hausdorff dimension of the inverse image
$X^{-1}(F) = \{t \ge 0: X(t) \in F\}$ and proved that if  $1 \le \alpha \le 2$ 
and $F \subseteq \R$
is a fixed Borel set, then for every $x \in \R$,
\begin{equation}  \label{Eq:Hawkes}
 \dimh X^{-1}(F)= 1 - \frac1\al + \frac{\dimh F}\al, \qquad  \PP^x\hbox{-a.s.}
\end{equation}
Here $\dimh $ denotes Hausdorff dimension; see Falconer \cite{Fal},
or \cite{Taylor86,Xiao04survey} for the definitions and properties of
Hausdorff measure and Hausdorff dimension.

Note that the null event on which (\ref{Eq:Hawkes}) does not hold depends 
on $F$. It is natural to ask if the following uniform Hausdorff dimension 
result holds: For every $x \in \R$,
\begin{equation}  \label{Eq:Hawkes1}
\PP^x \bigg(\dimh X^{-1}(F)= 1 - \frac1\al + \frac{\dimh F}\al
\hbox{ for all Borel sets }\, F \subseteq \R \bigg) =1.
\end{equation}
Such a result, when it is valid,  is more useful than (\ref{Eq:Hawkes}) because, 
outside of a single null event,  the dimension formula holds not only for all 
deterministic Borel sets $F \subset \R$  but also for random sets $F$ that 
depend on the sample path of $X$.

We claim that, in the case $0<\alpha<1$, there  is no uniform result like \eqref{Eq:Hawkes1}.
This is because $X^{-1}(F) = \emptyset$ $\PP^x$-a.s. if $\dimh F < 1- \al$.
%
%
The referee has asked us the following 
question that complements the aforementioned claim:\footnote{We thank the anonymous referee for 
this interesting question. Since $X^{-1}(F) = \emptyset$ $\PP^x$-a.s. if $\dimh F < 1- \al$, we have 
modified slightly the referee's question.} For every $x\in \R,$ does 
\begin{equation} \label{Eq:Hawkes3}
\PP^x \bigg( \dimh X^{-1}(F)= 1 - \frac1\al + \frac{\dimh F} \al  \, \hbox{ for all } \, F \in {\mathcal C}\bigg) < 1?
\end{equation}
Here ${\mathcal C}$ is the family of all \underline{deterministic} Borel sets $F \subset \R$ with 
$\dimh F \ge  1- \al$.  
To answer this question, we first recall Theorem 2 of Hawkes \cite{Hawkes71b} : 
If $0 < \al < 1$ and $F \subset \R$ is deterministic and satisfies $\dimh F \ge 1 - \al$,
then
\begin{itemize}
\item[(i)]\, For every $x \in \R$,
\begin{equation} \label{Eq:Hawkes2a}
\sup\Big\{\theta: \, \PP^x\big(\dimh X^{-1}(F) \ge \theta\big) > 0\Big\} = 1 - \frac 1 \al + \frac{\dimh F}\al.
 \end{equation}
\item[(ii)]\, If $x \in F^*$ (see \cite[p.93]{Hawkes71b} for the notation), then
\begin{equation} \label{Eq:Hawkes2b}
\PP^x \bigg( \dimh X^{-1}(F)= 1 - \frac1\al + \frac{\dimh F}\al \bigg) = 1
\end{equation}
\item[(iii)]\, If $F \backslash F^*$ is polar, then for  every $x\in \R,$
\begin{equation} \label{Eq:Hawkes2c}
\PP^x \bigg(\dimh X^{-1}(F)= 1 - \frac1\al + \frac{\dimh F}\al \Big| X^{-1}(F)
\ne \emptyset \bigg) =1.
\end{equation}
\end{itemize}
The answer to the referee's question is ``yes" because we can 
choose a Borel set $F \in {\mathcal C}$ such that $F \backslash F^*$ is polar for $X$ 
(cf. \cite[p.96]{Hawkes71b}), then it follows from 
Hawkes' result (iii) that  for any $x \in \R$ the probability in (\ref{Eq:Hawkes3}) is not more than
\[
\PP^x \bigg(\dimh X^{-1}(F)= 1 - \frac1\al + \frac{\dimh F}\al \bigg) = \PP^x \Big( X^{-1}(F)
\ne \emptyset \Big) < 1.
\]
Motivated by the referee's question and Hawkes' result (\ref{Eq:Hawkes2b}), one may further ask to 
characterize the family ${\mathcal G}$ of deterministic Borel sets $F$ such that for some $x \in \R$ 
(depending on ${\mathcal G}$),
\begin{equation} \label{Eq:Hawkes4}
\PP^x \bigg( \dimh X^{-1}(F)= 1 - \frac1\al + \frac{\dimh F} \al  \, \hbox{ for all } \, F \in {\mathcal G}\bigg) = 1.
\end{equation} 
This question seems to be rather nontrivial. We can imagine that (\ref{Eq:Hawkes4}) may hold for certain
family of self-similar sets on $\R$, but this goes beyond the scope of the present paper.

Our objective of this paper is to study the uniform dimension problem (\ref{Eq:Hawkes1}) for $1 \le \al \le 2.$ 
The validity of (\ref{Eq:Hawkes1})  in the case $\al = 2$ ($X$ is a Brownian motion)
is due to Kaufman \cite{Kaufman85}. His proof relies on the uniform modulus of continuity of
Brownian motion as well as the H\"older continuity of the Brownian local time in the time variable.
For $1 \le \alpha < 2$,  the sample paths of an $\alpha$-stable L\'evy process are discontinuous,
hence Kaufman's method is not applicable.

In the special case of $F = \{z\}$, it follows from Barlow et al \cite[(8.7)]{BPT86} that if $1 < \alpha \le 2$ then
\begin{equation}  \label{Eq:BPT}
\PP^x \Big(\dimh X^{-1}(z)= 1 - \frac1\al    \hbox{ for all  }\, z \in \R \Big) =1.
\end{equation}
This gives a uniform Hausdorff dimension result for the level sets of $X$. However,
for $1 \le \alpha < 2$,
it had been an open problem to prove (\ref{Eq:Hawkes1}) for all Borel sets $F \subseteq \R$;
see \cite[Sec. 8.2]{Xiao04survey} for a discussion.

In this note, we verify  (\ref{Eq:Hawkes1}) by proving the following theorem.

\begin{theorem}\label{theo}
Let $X$ be a real-valued strictly $\al$-stable L\'evy process  with $1< \al\le 2$.
For every $x\in\RR$, (\ref{Eq:Hawkes1}) holds.
\end{theorem}

As mentioned above, the case of $\alpha= 2$ has already been proved by Kaufman \cite{Kaufman85}
whose proof relies on special properties of Brownian motion. Our proof of  Theorem \ref{theo} provides
an alternative proof of his theorem.

The proof is split naturally into the upper bound part and lower bound part. To show the upper  bound,
we design a new covering principle (see Lemma \ref{l:covering} below) for the inverse images of recurrent
processes (thus it is applicable to $\alpha = 1$).  This covering lemma constitutes the key technical
contribution of the present paper, and we expect
 it to be useful for other discontinuous Markov processes.
Note that Lemma \ref{l:covering} in this paper is different from the covering lemma of \cite[Lemma 2.2]{SXXZ},
which is only applicable to transient Markov processes
(see Remark \ref{Re:cp} in Section 2 of this paper).  To prove the
lower bound in (\ref{Eq:Hawkes1}), we make use of the uniform modulus of continuity (in time) of the maximum
local time of $X$ due to Perkins \cite{Perkins86}, together with a covering principle for the range of $X$ in
\cite{HP74,Xiao04survey, SXXZ}. Since $X$ has no local time when  $\alpha = 1$, the proof of the lower bound
in Theorem \ref{theo} is valid only for $1< \al\le 2$. We think that (\ref{Eq:Hawkes1}) holds for  $\alpha = 1$
as well, but have not been able to give a complete proof.

\section{Proof of the upper bound}

In this section we assume that $1\le \al\le 2$. We will show that
\begin{equation}  \label{Eq:up}
\PP^x \bigg(\dimh X^{-1}(F) \le 1 - \frac1\al + \frac{\dimh F}\al
\hbox{ for all Borel sets }\, F \subseteq \R \bigg) =1.
\end{equation}

For any Borel set $B$, we denote by $T_B$ the first hitting time of $B$ by the process $X$.
We state an asymptotic result due to Port \cite[Thm. 2 and Thm. 4]{Port67}
on the first hitting time of compact sets by recurrent strictly stable processes, see \cite[Thm. 22.1]{PortStone71AIF}
for similar results in a more general setting.
Note that when $1\le \al\le 2$,  $X$ is recurrent by the Chung-Fuchs criterion (\cite[Thm. 16.2]{PortStone71AIF}),
and  any nonempty set has positive capacity, so the condition in \cite[Thm. 22.1]{PortStone71AIF} is satisfied.

\begin{lemma}\label{PS}
\begin{itemize}
\item[(1).] If $1<\al\le 2$, then for any bounded interval $B$ and any  $x\in \RR$,
\begin{align*}
\PP^x(T_B>t) \sim L_B(x)t^{-1+\frac{1}{\al}}, \qquad \mbox{ as } t\to\infty,
\end{align*}
where $L_B(x)$ is bounded from above on compact sets and is positive for $x\not\in \ov B$, the closure of the
set $B$. Here, $f(t)\sim g(t)$ means $\lim_{t\to\infty} f(t)/g(t)=1$.
\item[(2).] If $\alpha=1$, then for any bounded interval $B$ and any $x\in \RR$,
\begin{align*}
\PP^x(T_B>t) \sim
\frac {L_B(x)} {\log  t}, \qquad \mbox{ as } t\to\infty,
\end{align*}
where $L_B(x)$ is bounded from above on compact sets and is positive for $x\not\in \ov B$.
\end{itemize}
\end{lemma}

  The main tool to obtain our upper bound is the following covering lemma.
Before stating this lemma, we introduce some notation.
Let $\mathcal U_n$ be any partition of $\RR$ with intervals of length $2^{-n}$ and $\mathcal D_n$ be any partition
of $\RR_+$ with intervals of length $2^{-n\al}$.
The choices of partitions have no effect on the result.

\begin{lemma} \label{l:covering}
 Let $1\le\al\le 2$. Let $\de> \al-1$ and $T>0$.
$\PP^x$-a.s., for all $n$ large enough and every $U\in\mathcal U_n$, $X^{-1}(U)\cap[0,T]$ can
be covered by $2\cdot 2^{n\de}$ intervals from $\mathcal D_n$.
\end{lemma}

\begin{proof}
(1) Suppose first $1<\al\le 2$.
For a fixed interval $U\in\mathcal U_n$, write $U=(z- \frac{2^{-n}}{2},z+\frac{2^{-n}}{2})$ for some $z\in\RR$.
Let $\tau_0=0$ and, for all $k\ge 1$, define
\begin{align*}
\tau_k = \inf\bra{ s>\tau_{k-1} + 2^{-n\al}:  |X(s)-z| < \frac{2^{-n}}{2} },
\end{align*}
with the convention that $\inf\emptyset=\infty$. It is clear that
$X^{-1}(U)\subset \bigcup_{i=0}^\infty [\tau_i, \tau_i+2^{-n\al}]$,
which implies that
\begin{align*}
\bra{\tau_k\ge T} \subset \bra{X^{-1}(U)\cap[0,T] \mbox{ can be covered by } k \mbox{ intervals of length } 2^{-n\al}}.
\end{align*}
Therefore,
\begin{align*}
\bra{X^{-1}(U)\cap[0,T] \mbox{ cannot be covered by } k \mbox{ intervals of length } 2^{-n\al} } \subset \bra{\tau_k< T}.
\end{align*}
Note by spatial homogeneity and scaling, we have  that
$$\PP^x\left( \inf_{2^{-n\al}\le s\le T} |X(s)-x|\le 2^{-n}
\right)= \PP^0\left( \inf_{1\le s\le T2^{n\al}} |X(s)| \le 1 \right):=p_n.$$
Due to the right continuity of the sample paths,
we have $X(\tau_{k-1})\in \ov U$ as $\tau_{k-1}<T$.  By the strong Markov property,  we obtain
\begin{align*}
\PP^x(\tau_k< T) &= \PP^x(\tau_k< T|\tau_{k-1}< T)\PP^{x}(\tau_{k-1}\le T) \\
&\le \sup_{y\in \ov U} \PP^y\left( \inf_{2^{-n\al}\le s\le T} |X(s)-z|\le 2^{-n}/2
\right) \PP^x(\tau_{k-1}\le T)\\
&\le \sup_{y\in \ov U} \PP^y\left( \inf_{2^{-n\al}\le s\le T} |X(s)-y|-|y-z|\le 2^{-n}/2
\right) \PP^x(\tau_{k-1}\le T)\\
&\le \sup_{y\in \ov U} \PP^y\left( \inf_{2^{-n\al}\le s\le T} |X(s)-y| \le 2^{-n}
\right) \PP^x(\tau_{k-1}\le T) \\
&=  p_n \cdot \PP^x(\tau_{k-1}\le T).
\end{align*}
By induction, we obtain $$\PP^x(\tau_k< T) \le p_n^k.$$

\sk
Next we show that
there exists a constant $c_T$ such that $p_n\le 1-c_T 2^{-n\al(1-\frac 1 \al)}$.
By the independence of increments and the fact that $X(1)$ is supported on $\RR$ (\cite[Thm. 1]{Taylor67}),
\begin{align*}
1-p_n &\ge \PP^0(2\le X(1)\le 3, \inf\{t\ge 1: X(t)-X(1) \in [-4,-1]\}\ge T 2^{n\al} ) \\
&\ge c \,\PP^0(T_{[-4,-1]}\ge T2^{n\al}).
\end{align*}

 Lemma \ref{PS} implies that
\begin{align*}
1-p_n \ge c_T 2^{-n\al(1-\frac 1 \al)},
\end{align*}
as desired.
For $n,K\ge 1$,  define the event $A^\de_n$ by
\begin{align*}
\bra{ \exists U\in\mathcal U_n\cap [-K,K], \mbox{ s.t. } X^{-1}(U)\cap[0,T] \mbox{ cannot be  covered by }
2^{n\de} \mbox{ intervals of length }2^{-n\al}}.
\end{align*}
Here $U\in\mathcal U_n\cap[-K,K]$ means that $U\in\mathcal U_n$ and $U\subset [-K,K]$.  We have
for $\de>\al -1$,
\begin{align*}
\sum_{n=1}^\infty \PP^x(A_n^{\de}) &\le \sum_{n=1}^\infty \sharp\{U\in\mathcal U_n: U\cap[-K,K]
\neq\emptyset\} (p_n)^{2^{n\de}} \\
& \le 2K\sum_{n=1}^\infty  2^{n} (1-c_T 2^{-n\al(1-\frac 1 \al)})^{2^{n\de}} \\
&\le 2K\sum_{n=1}^\infty \exp\left(n(\log 2) - c_T2^{n(\de-\al+1)}\right) < \infty.
\end{align*}
Since any interval of length $2^{-n\al}$ is covered by two intervals from $\mathcal D_n$,
the conclusion for all $U\subset [-K,K]$ follows from the Borel-Cantelli Lemma.
Letting $K\to \infty$ completes the proof.

(2) Now consider $\al=1$.  The proof of this case is basically the same as that of Part (1), except that $1-p_n
\ge c_T /n$ by Lemma \ref{PS}.(2), and
\begin{align*}
\sum_{n=1}^\infty \PP^x(A_n^{\de}) \le 2K\sum_{n=1}^\infty \exp(n(\log 2) - c_T 2^{n\de}/n ) <\infty.
\end{align*}
We omit the details.
\end{proof}

\begin{remark}\label{Re:cp}
As is said in the Introduction,  the covering principle in \cite[Lemma 2.2]{SXXZ} is not applicable here.
Intuitively,  a recurrent process visits a fixed interval infinitely often, hence we could not expect that
the inverse images could be covered by finite number of intervals.   Mathematically,
the condition in \cite{SXXZ} is
\begin{align*}
\PP^x\left(\inf_{t_n\le t<T} |X(s)-x|\le r_n\right) \le K r_n^\de
\end{align*}
for some $\de,p>0$ and $\sum_{n=1}^\infty r_n^p<\infty$,  which is not satisfied for recurrent Markov processes.
\end{remark}

Let us prove the upper bound (\ref{Eq:up}).

\begin{proof}[Proof of Theorem \ref{theo}: upper bound.]
We first consider the case $1 < \alpha \le 2$.
For any Borel set $F$, let $\theta>\dimh F$ and $\de>\al-1$. Then there exists a sequence of intervals $\bra{U_i}$
of length $2^{-n_i}$ such that
\begin{align*}
F\subset \bigcup_{i=1}^\infty U_i \quad \mbox{ and } \quad \sum_{i=1}^{\infty} 2^{-n_i\theta} <1.
\end{align*}
Fix a $T>0$ for now.
By Lemma \ref{l:covering},
each $X^{-1}(U_i)\cap[0,T]$ can be covered by $2\cdot 2^{n_i\de}$ intervals $\bra{I_{i,k}}$
(of length $2^{-n_i\al}$) in $\mathcal D_{n_i}$, thus we see that
\begin{align*}
X^{-1}(F) \cap [0,T] \subset \bigcup_{i=1}^\infty\bigcup_{k=1}^{2\cdot 2^{n_i\de}} I_{i,k}.
\end{align*}
Moreover, let $d=(\theta+\de)/\al$,
\begin{align*}
\sum_{i=1}^{\infty} \sum_{k=1}^{2\cdot 2^{n_i\de}} [\mbox{diam}(I_{i,k})]^d = 2 \cdot\sum_{i=1}^{\infty} 2^{n_i\de}2^{-n_i\al d}
= 2\cdot \sum_{i=1}^{\infty} 2^{-n_i\theta}< 2.
\end{align*}
This proves $\dimh X^{-1}(F)\cap [0,T] \le d$. Letting $\theta\downarrow \dimh F$, $\de\downarrow (\al-1)$ and
$T\uparrow \infty$ yields the desired upper bound.

Now we consider the case of $\alpha = 1$. One could repeat the argument above and use
Lemma \ref{l:covering} to get the desired conclusion. Here we present an alternative argument.
It follows from Hawkes and Pruitt \cite{HP74} (see also \cite{SXXZ}) that the following uniform dimension result holds:
\begin{equation}\label{Eq:HP74}
 \PP^x\left( \dimh X (E) = \dimh E  \ \mbox{ for all Borel } E\subset\RR_+\right)=1.
 \end{equation}
For any Borel set $F\subset\RR$, let $E = X^{-1}(F)$. Then $X(E) \subseteq F$. On the event in \eqref{Eq:HP74},
we have $\dimh E = \dimh X(E) \le \dimh F$.
Hence, $\PP^x\left(\dimh X^{-1}(F) \le \dimh F\,  \mbox{ for all } F\subset\RR\right)=1.$
\end{proof}

\section{Proof of the lower bound}

We assume that $1 < \al \le 2$. It follows from Kesten \cite{kesten69} and Hawkes \cite{hawkes86}
that $X$ hits points and has local times $\{L_t^x, t \ge 0, x \in \RR\}$.
The local times
characterize the sojourn properties of $X$  via the occupation density formula:  For all $t\ge 0$ and
all Borel measurable function $f:\RR\to\RR$,
\begin{align*}
\int_0^t f(X(s))\rd s = \int_{\RR} f(x)L^x_t\rd x.
\end{align*}
Moreover,  there is a version of the local times, still denoted by $\{L_t^x, t \ge 0, x \in \RR\}$,
which is jointly continuous in $(t, x)$; see e.g., \cite{Bertoin96book, MarcusRosen06}.

We use  the H\"older continuity of the local times of $X$ to prove the uniform lower bound for the
inverse image sets.  This approach has been previously used by Kaufman \cite{Kaufman85},
which was extended by Monrad and Pitt \cite{MonradPitt86} in their study of inverse images
of recurrent Gaussian fields. In both articles,  the uniform modulus of continuity of the sample
paths were used.
Since the sample paths of the $\alpha$-stable L\'evy process $X$ are discontinuous, we will
apply a covering principle in \cite{Xiao04survey, SXXZ} for
the range of $X$.
Denote $\mathcal C_n$ any partition of $\RR_+$ of intervals of length $2^{-n}$.
We recall here the covering principle, tailored to our situation.

\begin{lemma}\label{lem:cov_range}
Let $0<\ga<\frac 1 \al$. There exists a finite positive integer $K$, such that $\PP^x$-a.s.,
for all $n$ large enough, $X(I)$ can be covered by $K$
intervals of diameter $2\cdot 2^{-n\ga}$, for all $I\in\mathcal C_n$.
\end{lemma}
\begin{proof}
It suffices to verify condition (2.1) in the statement of \cite[Lem. 2.1]{SXXZ}, namely,
there exist $\de>0$ and $K_0<\infty$ such that
\begin{align*}
\PP^x \Big(\sup_{0\le s\le 2^{-n}}|X(s)-x|\ge 2^{-n\ga} \Big) \le K_0 2^{-n\de}.
\end{align*}
By spatial homogeneity and scaling, the probability above is equal to
\begin{align*}
\PP^0 \Big(\sup_{0\le s\le 1}|X(s)|\ge 2^{n(\frac 1 \al -\ga)}\Big),
\end{align*}
which, by \cite[Thm. 5.1]{BottcherSchillingWang13}, is bounded from above by $2^{-n\de}$
with $\de=1-\ga\al$, as desired.
\end{proof}

\mk
Let $L^*({[s,t]})= \sup_{x\in\RR} (L^x_t - L^x_s)$
be the maximum local time of $X$ on $[s, t]$.
We recall now the following result due to Perkins \cite{Perkins86} on the
uniform modulus of continuity (in time) of the maximum local time of a strictly
$\al$-stable L\'evy process $X$ with index $\alpha \in (1, 2]$.

\begin{lemma}\label{lem:mod_LT}  There exists a finite positive constant $c_1$ such that
\begin{equation}\label{Eq:Uni-maxLT}
\limsup_{r\to 0}\sup_{|s-t|<r \atop 0\le s<t\le 1} \frac{L^*({[s,t]})}{r^{1 - \frac 1 \al}
(\log 1/r)^{\frac 1 \al}} = c_1, \quad \PP^x\mbox{-a.s.}
\end{equation}
\end{lemma}

We refer to  Ehm \cite[Thm. 2.1]{Ehm81}
or Khoshnevisan, Zhong and Xiao \cite[Thm. 4.3]{KhoshnevisanXiaoZhong03SPA} for related
results; and to Marcus and Rosen \cite{MarcusRosen92, MarcusRosen92p,MarcusRosen06}
for more sample path properties (in the space variable) of the local times of symmetric Markov processes.

\sk

We are ready to give the proof of the lower bound in Theorem \ref{theo}.

\begin{proof}[Proof of Theorem \ref{theo}: lower bound.]
It suffices to consider compact set $F$. For any compact $F\subset \RR$ and $\ep>0$, by
Frostman's lemma (cf. \cite{Fal}) there exists a probability measure $\mu$ supported
 on $F$ such that $\mu(B)\le |\mbox{diam}(B)|^{\dimh F-\ep}$ for any interval $B\subset \RR$
 with $|B|\le 1$.  Define the random measure $\la$ by
\begin{align}\label{Def:la}
\la([a,b]) =  \int_\RR (L^x_b - L^x_a)\mu(\rd x)  \quad \mbox{ for } 0\le a\le b.
\end{align}
It is clear that $\la(\rd t)$ is supported on $X^{-1}(F)\subset \RR^+$, $\la(\RR^+)>0$, and
\begin{equation*}
\la([a,b]) \le L^*([a,b]) \mu(\ov{X([a,b])}).
\end{equation*}
Let $n$ be sufficiently large, we have by Lemma \ref{lem:mod_LT} that
 $$L^*([a,a+2^{-n}])\le 2^{-n(1-\frac 1 \al -\ep)} $$
uniformly for $a\in[0,1-2^{-n}]$. On the other hand, by Lemma \ref{lem:cov_range},
there exist a sequence of intervals $\{I_i\}_{1\le i\le K}$ of length $2^{-n\ga}$ with $\ga<1/\al$
such that the closure of ${X([a,a+2^{-n}])}$ is covered by the union of $I_i$, therefore,
 \begin{align}\label{Eq:cov-upp}
 \mu(\ov{X([a,a+2^{-n}])}) \le \sum_{i=1}^K \mu(I_i) \le K 2^{-n\ga(\dimh F -\ep)}.
 \end{align}
 We thus obtain
 \begin{align*}
 \la([a,a+2^{-n}]) \le K 2^{-n(1-\frac 1 \al + \ga\dimh F -2\ep)}.
 \end{align*}
It follows that $\la(B)\le \mbox{diam}(B)^{1-\frac 1 \al + \ga\dimh F -2\ep}$ for all Borel sets $B$ with sufficiently
small diameter.  This and Frostman's lemma imply that
\begin{align*}
\PP^x \bigg(\dimh X^{-1}(F)\ge 1-\frac 1 \al + \ga\dimh F -2\ep \mbox{ for all compact Borel  } F\bigg)=1.
\end{align*}
Letting $\ga\uparrow \frac 1 \al$, then $\ep\downarrow 0$ yields the desired lower bound for $\dimh X^{-1}(F)$.
This finishes the proof of Theorem \ref{theo}.
\end{proof}
\mk

\section{Concluding remarks}
This note
raises several interesting questions for further investigation. In the following, we list three of them and discuss briefly
the main difficulties. Solutions of these questions will require developing new techniques for
L\'evy processes.
\begin{itemize}
\item[(i).]\,As having mentioned in the Introduction, we think that
Theorem \ref{theo} holds for  $\alpha = 1$. However, without a local time, it is not clear to us how to construct a random
Borel measure supported on $X^{-1}(F)$ such that  Frostman's lemma is applicable.
\item[(ii).]\, In  \cite[Thm. 22.1]{PortStone71AIF}, the asymptotic result for the hitting times was obtained for recurrent
L\'evy processes with regularly varying $\la$-potential densities, see also the recent development by Grzywny and
Ryznar \cite{GrzywnyRyznar17}. Our method for proving the upper bound of $\dimh X^{-1}(F)$ is still applicable if the
characteristic exponent of $X$ is regularly varying at zero with index $\al\in (1,2]$. On the other hand, by modifying
the methods in Ehm \cite{Ehm81}, Khoshnevisan, Zhong and Xiao \cite[Thm. 4.3]{KhoshnevisanXiaoZhong03SPA},
we can prove an upper bound for the uniform modulus of continuity in the time variable for the maximum local time as
the one in Lemma \ref{lem:mod_LT}  for L\'evy processes with regularly varying exponent $\al\in (1,2]$. Hence,
Theorem \ref{theo} is valid for L\'evy processes with regularly varying exponent $\al\in (1,2]$. We believe that a similar
result also holds for a large class of more general Markov processes including stable jump diffusions, stable like processes
and L\'evy-type processes as considered in \cite{SXXZ}. However, proving
such a result would require establishing first the asymptotic results for the hitting times and local times of these Markov processes.
This is pretty challenging and  goes well beyond the scope of the present paper.
We will try to tackle this in a subsequent paper.
\item[(iii).]\, It is natural to expect that the packing dimension analogue of Theorem \ref{theo}  also holds.
Namely, if $X$ is a real-valued strictly $\al$-stable L\'evy process
with $1\le \al\le 2$, then for any $x\in\RR$ one has
\begin{equation}  \label{Eq:packing}
\PP^x \bigg(\dimp X^{-1}(F)= 1 - {1\over \al} + {\dimp F\over \al}  \hbox{ for all Borel sets }\, F \subseteq \R \bigg) =1.
\end{equation}
Here $\dimp $ denotes packing dimension; see Falconer \cite[Chapter 3]{Fal} for its definition and properties, and
\cite{Taylor86,Xiao04survey} for examples of its applications in studying sample path properties of Markov processes.

By using the connection between packing dimension and the upper box-counting (Minkowski) dimension (cf. \cite{Fal}),
one can see that the proof of the upper bound of Theorem \ref{theo} also implies that $\PP^x$-a.s.,
\[
\dimp X^{-1}(F)\le 1 - {1\over \al} + {\dimp F\over \al}  \hbox{ for all Borel sets }\, F \subseteq \R.
\]
In order to prove the reverse inequality, one may apply the lower density theorem for packing
measure in \cite[Theorem 5.4]{TT85} and prove  that for any $\gamma < 1/\alpha$ and $\varepsilon > 0$,
$$  \sup_{a \in X^{-1}(F)}  \liminf_{r \to 0} \frac{\lambda([a, a+r] )} {r^{1 - \alpha^{-1} + \gamma \dimp F - 2 \varepsilon}} \le c_2 < \infty,$$
where $\lambda$ is the random measure defined in \eqref{Def:la} and $c_2$ is a finite constant. We are not able to prove
this because (unlike the Hausdorff dimension case) the terms $\mu(I_i)$ in \eqref{Eq:cov-upp} can not be controlled
for all $i$ by the same $n$.
\end{itemize}

\bigskip
\noindent
{\bf Acknowledgements:}
We thank the referee for carefully reading the
manuscript and providing some useful suggestions.

\bigskip
\noindent

\bibliographystyle{plain}

\end{document}